\documentclass[preprint,11pt]{elsarticle}
\usepackage{amsmath,amsthm,amscd,amssymb,latexsym,graphicx}
\usepackage{geometry}                
\usepackage[colorlinks,citecolor=red,pagebackref,hypertexnames=false]{hyperref}
\usepackage{comment}

\newcommand\RR{{\mathbb{R}}}
\newcommand\R{{\mathbb{R}}}
\newcommand\CC{{\mathbb{C}}}

\newcommand\ZZ{\mathbb{ Z}}

\numberwithin{equation}{section}
\newtheorem{theorem}{Theorem}[section] 
\newtheorem{lemma}[theorem]{Lemma}
\newtheorem{corollary}[theorem]{Corollary}

\newtheorem{definition}[theorem]{Definition}

\newtheorem{example}[theorem]{Example}

\newtheorem{remark}[theorem]{Remark}

\newtheorem{notation}[theorem]{Notation}

\newtheorem{case[theorem]}{Case}

\begin{document}

\begin{frontmatter}

\title{Exponential bases, Paley-Wiener spaces and applications}


\author[AI]{Alex Iosevich}

 \address[AI]{University of Rochester University, Rochester, NY 14627, \emph{iosevich$@$math.rochester.edu}} 
 
\author[AM]{Azita  Mayeli (corresponding author) \footnote{This research was partially supported by PSC-CUNY Grant  67736-0045.}}
 
 \address[AM]{City University of New York, The Graduate Center, NY 10016, \emph{amayeli$@$ gc.cuny.edu}}
 
\begin{abstract}  We investigate the connection between translation bases for Paley-Wiener spaces and exponential Fourier bases for a domain. We apply these results to the characterization of vector-valued time-frequency translates of a Paley-Wiener \lq\lq{}window\rq\rq{} signal. 


     
 \end{abstract}    
 

 \end{frontmatter}            

\section{Introduction}\label{Introduction}


The main goal of this paper is to develop a correspondence between exponential basis on a domain in the Euclidean space and   translation bases for Paley-Wiener spaces.   
 Let $\Omega$ be a  Borel  subset of ${\Bbb R}^d$, $d \ge 1$, with positive and finite  measure. We define the Paley-Wiener space $PW_\Omega({\Bbb R}^d)$   on domain $\Omega$ to be the  subspace of  all $L^2({\Bbb R}^d)$ functions $f$ whose Fourier transform are supported in $\Omega$. More precisely, 
 
$$ PW_\Omega(\R^d):= \{ f\in L^2(\R^d): \  \hat f(\xi)= 0  \  \text{for} \ a.e.\  \xi\in \Omega^c\}.$$ 
 
  For $f\in  PW_\Omega(\R^d)$, we say $\hat f$ is nowhere zero  on $\Omega$ if $\hat f(\xi)\neq 0$ for all $\xi\in \Omega$. 
Given  a countable  set $A$ in $\Bbb R^d$ and $\phi \in L^2({\Bbb R}^d)$, we  say the system of  translates    ${ \{ \phi(\cdot-a) \} }_{a \in A}$  is a  Riesz basis for $PW_\Omega(\Bbb R^d)$     if  there exist positive and finite  numbers $C_1,C_2$  such that    for any  $f \in PW_\Omega({\Bbb R}^d)$   there is  ${ \{c_a \}}_{a \in A}\in l^2(A)$  such that 
$$ f(x)=\sum_{a \in A} c_a \phi(x-a) $$
 and  
\begin{align}\label{shift-frame-inequality}
C_1 \sum_{a \in A} |c_a|^2 \leq {||f||}_2^2 \leq C_2 \sum_{a \in A} |c_a|^2.
\end{align}
In this case, the constants $C_1$ and $C_2$ are called Riesz constants. \\

 
 For the rest of this paper, we denote by $A$ a countable set in $\RR^d$ and 
 by $\Omega$  a Borel set of positive and finite measure. And, by $\hat\phi$ we mean the Fourier transform of 
 $\phi$, defined, as usual by 
 $$ \widehat{\phi}(\xi)=\int e^{-2 \pi i x \cdot \xi} \phi(x) dx.$$


   We say  set $A$   is a  {\it Riesz basis  spectrum} for $L^2(\Omega)$ 
 if there exist positive and finite constants  $C_1,C_2$ such that  for any  $f \in L^2(\Omega)$, there is    $\{c_a\}_{a\in A}\in l^2(A)$ such that 
$$ f =\sum_{a \in A} c_a e_a,  $$  
and 
\begin{align}\label{exponential-frame-inequality}
  C_1 \sum_{a \in A} |c_a|^2 \leq {||f||}_{L^2(\Omega)}^2 \leq C_2 \sum_{a \in A} |c_a|^2. 
  \end{align}
Here, $e_a(x) :=  e^{-2 \pi i \langle x , a\rangle}$, $a\in A$, $x\in \Omega$, and $\langle x, a\rangle$ is the inner product in 
$\RR^n$. In this case we call $\Omega$ a {\it Riesz basis spectral set}.  Analogously, we say $A$ is an {\it orthonormal basis spectrum} and $\Omega$ is {\it orthonormal basis spectral set}  if $\{e_a\}_{a\in A}$ is an orthonormal basis for $L^2(\Omega)$. See the seminal paper by Fuglede (\cite{Fu74}) for background on orthonormal exponential bases. \\


In many problems in harmonic analysis, approximation theory, wavelet theory, sampling theory, Gabor systems, and signal processing \cite{BVR1, AG, HW, M, RS95, RS97}, it is interesting to construct a function $\phi\in L^2(\Bbb R^d)$, such that the set  $\{\phi(\cdot-a)\}_a$ is an orthonormal or a Riesz sequence. The structure of $\phi$ for which $\{\phi(\cdot-j)\}_{j\in \ZZ^d}$  is   an orthonormal basis or  Riesz basis  for the closure of its spanned space  in $L^2(\Bbb R^d)$ topology  has been extensively studied in,  for example, \cite{BVR1, BVR2, HSWW, B}.    
The techniques used   in these papers strongly rely on the periodic tiling property of $[0,1]^d$ for $\Bbb R^d$. In general, the same techniques can not be used for the characterization of general translates of $\phi$   when 
 $\Bbb Z^d$  and $[0,1]^d$ are replaced by any countable set $A$ and  any  Borel set $\Omega$, respectively. In this paper we extend the results of the characterization of translate bases where  $A$ {\it does not} necessarily have a group structure.  In our paper, we shall focus our attention only on the class of Paley-Wiener functions $\phi \in PW_\Omega(\Bbb R^d)$. \\
 
This paper is structured as follows. The first main result is stated in  
Theorem \ref{equivalence2} in Section \ref{Introduction}. In Section \ref{first part} we prove Theorem   \ref{equivalence2} and we establish a correspondence between bases of the type $\{\phi(\cdot-a)\}_{a\in A}$ in $L^2(\Bbb R^d)$ and spectrum $A$ for $\Omega$. Moreover, we prove that under some mild assumptions on $\hat\phi$, the necessary and sufficient conditions for $\{\phi(\cdot-a)\}_{a\in A}$  to be a Riesz basis for $PW_\Omega(\Bbb R^d)$ is that $A$ is a Riesz basis  spectrum for $\Omega$. In Section \ref{second part}, we extend the results of Section  \ref{first part}   to general bases, i.e.,  frames. In Section \ref{Paley-Wiener valued Gabor systems}, we use our results to characterize the structure of Paley-Wiener valued Gabor systems. In particular, we use  the result  of Theorem \ref{equivalence2} on the  characterization of exponential functions  in the proof of Theorem \ref{vector-valued bases},   thus illustrating the connection between vector-valued  Gabor frames and exponential bases. 
\\
  


We shall recall the following result about the orthonormal bases of exponentials on a domain. 
 
\begin{theorem} \label{equivalence}  Let $|\Omega|=1$. For $\phi\in L^2(\Bbb R^d)$, the system   $\{ \phi(\cdot-a) \}_{a \in A}$ is an orthonormal basis for $PW_\Omega(\Bbb R^d)$ if and only if $\{e_a\}_{a\in A}$ is an orthonormal basis for $L^2(\Omega)$, provided that   $|\hat\phi(x)|=\chi_\Omega$  for a.e. $x\in \Omega$. 
\end{theorem} 
 
 
Note that by the above theorem we can recover the  Whittaker-Shannon-Kotel\rq{}nikov Theorem which states 
that the sequence 
$$\left\{sinc(x -n) = \cfrac{\sin\pi(x-n)}{\pi(x -n)}: \ n\in \Bbb Z \right\}$$ is an orthonormal basis for $PW_{[-1/2,1/2]}(\Bbb R)$,   the space of all functions in $L^2(\Bbb R)$  with Fourier support in $[-1/2,1/2]$. \\

 The following result proved by Lagarias, Reed and Wang (\cite{LRW00}) and, independently, by A. Iosevich and S. Pederson 
  (\cite{AP}) characterizes all orthonormal basis spectrum $A$ for  $d$-dimensional cube $\mathcal Q^d= [0,1]^d$. Their results are  equivalent to  those  in Theorem \ref{equivalence} when   $\Omega$ is the cube.

 \begin{theorem}[\cite{AP}]\label{AP} Let $A$ be a subset of $\Bbb R^d$. Then $A$ is an orthonormal basis spectrum 
 for the  $d$-dimensional unite cube $\mathcal Q^d$ if and only if $A$ is a tiling set for
  the cube $\mathcal Q^d$. 
 \end{theorem}

 For a function $u$, we let $\|u\|_\infty$ and $\|u\|_0$ denote the supremum and infinitum of $|u|$ on $\Omega$, respectively. Our next result can be viewed as a general version of Theorem \ref{equivalence}. 
 
\begin{theorem} \label{equivalence2} 
Let $\phi\in PW_\Omega(\Bbb R^d)$ and  $0<\|\hat\phi\|_0\leq \|\hat\phi\|_\infty<\infty$.   Then ${ \{ \phi(x-a) \} }_{a \in A}$ is  Riesz basis  for $PW_\Omega({\Bbb R}^d)$ if and only if $\{e_a\}_{a\in A}$ is   Riesz basis for $L^2(\Omega)$. In this case, the associated Reisz constants 
for $\{e_a\}_{a\in A}$  and
   ${ \{ \phi(x-a) \} }_{a \in A}$ are equal and 
  $C_1=\|\hat\phi\|_0$ and $C_2=\|\hat\phi\|_\infty$. 
\end{theorem}

As a consequence of Theorem \ref{equivalence2} we have the following results.

  \begin{corollary}\label{equivalence2-2} %
Let  $\phi\in PW_\Omega(\Bbb R^d)$ such that  $\hat \phi$ is continuous and nowhere zero  on $\Omega$.  If  $\Omega$ is compact,   then ${ \{ \phi(x-a) \} }_{a \in A}$ is  Riesz basis  for $PW_\Omega({\Bbb R}^d)$ if and only if $\{e_a\}_{a\in A}$ is Riesz basis for $L^2(\Omega)$.
 \end{corollary}


\begin{corollary} Let $A$ be a Riesz   spectrum for $\Omega$. Then for any $u\in L^2(\Omega)$, the set $\{u e_a\}_a$ is a Riesz basis for $L^2(\Omega)$ if and only if 
$0<\|u\|_0\leq \|u\|_\infty<\infty$. 
\end{corollary}

\section{Proofs  of Theorems \ref{equivalence} and \ref{equivalence2}}\label{first part}

\begin{proof}[{\bf Proof of Theorem \ref{equivalence}}.] The proof is straightforward, but we write it down for the sake of completeness. Assume that   $\{e_a\}$ is an orthonormal basis for  $L^2(\Omega)$.  
 Let $\phi\in PW_{\Omega}(\Bbb R^d)$ such that $|\hat\phi(x)|=1$ a.e. $x\in \Omega$. 
 We prove that 
 $\{L_a \phi\}$ is an orthonormal basis for $PW_\Omega(\Bbb R^d)$. 
  By the Parseval identity, the orthogonality of  $\{\phi(\cdot -a)\}$ is obtained by the orthogonality of  $\{e_a\}$ and that $|\hat\phi| = 1$ on $\Omega$: 
  
  \begin{align}\label{PI}
  \langle \phi(\cdot -a) , \phi(\cdot -a)\rangle =   \langle e_a, e_{a\rq{}}\rangle =\delta_{a, a\rq{}}. 
  \end{align}
 (Here, $\delta$ is  Kronecker delta.)
 Let $f$ be a function in $PW_\Omega(\Bbb R^d)$. Then  $\hat f\in L^2(\Omega)$ and hence 
 
 $$\hat f= \sum_a \langle \hat f, e_a\rangle   e_a = \sum_a \langle \hat f, \chi_\Omega e_a\rangle  \  \chi_\Omega e_a $$
 
 By the inverse Fourier transform and the Parseval identity  we have 
 $$ f = \sum_a \langle f, L_a \phi\rangle L_a \phi .$$

  Conversely, assume that $\{L_a\phi\}$ is an ONB for $PW_\Omega(\Bbb R^d)$. Again, the orthogonality of $\{e_a\}$ is obtained from the relation (\ref{PI}). Let $u\in L^2(\Omega)$. Then for $\check u$,  the inverse Fourier transform of $u$, we have 
  
  $$\check u = \sum_a \langle \check u, L_a \phi\rangle L_a\phi,$$
  
hence by the Fourier transform 
   $$u = \sum_a \langle u, e_a\rangle  e_a, $$
   
   as desired. 
 \end{proof}

For  a compact and symmetric convex domain $\Omega$ in the plane,   in \cite{AKT}, the first author, with Katz and Tao,  proved that a set $A$ is an orthonormal  basis spectrum for $\Omega$ if and only if $\Omega$ tiles $\Bbb R^2$ by translations. As corollary of their result along  Theorem \ref{equivalence},
in the following corollary we illustrate the relation between orthonormal translation bases and tiling property in $\RR^2$.

  \begin{corollary} 
  Let $\Omega$ be any compact convex domain in the plane and $A$ be a set in $\Bbb R^2$.  Then for any $\phi\in PW_\Omega(\Bbb R^d)$  with  Fourier transform  support in $\Omega$, the system of translations $\{\phi(\cdot-a)\}_A$ is orthonormal and complete in $PW_\Omega(\Bbb R^2)$ if and only if $A$ is a tiling set for $\Omega$, provided that $|\Omega|=1$ and $\hat \phi =\chi_\Omega$, a.e. 
  \end{corollary}


\begin{proof}[{\bf Proof of Theorem \ref{equivalence2}}] Assume that $A$ is a  Riesz  spectrum  for $\Omega$ and  $C_1$ and $C_2$ are the associated   Riesz constants. 
Let  $\phi\in PW_\Omega(\Bbb R^d)$ for which   $0<\|\hat\phi\|_0\leq \|\hat\phi\|_\infty<\infty$,   
and let   $f\in PW_\Omega(\Bbb R^d)$. Then put 
 $$u :=\cfrac{\hat f}{\hat \phi}. $$ 
  $u$ is in $L^2(\Omega)$ and for some  $\{c_a\}_{a\in A}\in l^2(A)$    
$$ u  = \sum_{a\in A} c_a e_a,  $$
such that 
 \begin{align}\label{exponential-frame-ineq}
 C_1 \sum_a |c_a|^2 \leq \|u\|^2_{L^2(\Omega)} \leq C_2  \sum_a |c_a|^2 .
 \end{align}
 
We will show  that $f(x) = \sum_a c_a \ \phi(x-a)$   in  $L^2(\Bbb R^d)$  and the inequality (\ref{shift-frame-inequality}) holds for  some other constants  
$0< C_1\rq{} \leq C_2\rq{}<\infty$.    We proceed  as following.  Let $B$ be a finite subset of $A$. Then 
\begin{align}\notag
\| f - \sum_{a\in B} c_a \phi(\cdot - a) \|_{PW_\Omega(\Bbb R^d)}^2 &= \| \hat f - \sum_{a\in B} c_a e_a \hat \phi\|_{L^2(\Omega)}^2
\\\notag
 &= \int_\Omega \left| \hat f(x) - \sum_{a\in B} c_a \hat\phi(x) e_a(x)\right|^2 dx\\\notag
  &= \int_\Omega \left| \hat\phi(x)\left(u(x) - \sum_{a\in B} c_a e_a(x)\right)\right|^2 dx\\\notag
  & \leq 
 \|\hat\phi\|_\infty^2 
 \|  u - \sum_{a\in B} c_a e_a \|^2 
\end{align}

Since the above inequality holds for any finite subset of $A$ and since   
$u = \sum_{a\in A} c_a e_a$, then   $f = \sum_{a\in A} c_a L_a\phi$ holds. It remains to show that the inequality (\ref{shift-frame-inequality}) holds for $f$.  Note that 
$\| f\|^2 = \|\hat f\|^2 = \|u \hat \phi \|^2$. Therefore 
by the upper estimation in 
 (\ref{exponential-frame-ineq})   
we  will have 

\begin{align}\notag
\| f\|^2  
\leq   \|\hat\phi\|_\infty^2\| u \|^2  
  \leq    C_2 
 \|\hat\phi\|_\infty^2 \sum_a |c_a|^2
\end{align}
 
With a similar argument, one can show that 
$$\| f\|^2 \geq   C_1   \|\hat\phi\|_0^2   \sum_a |c_a|^2.$$

Consequently we have 

\begin{align}\notag
 C_1\rq{}  \sum_a |c_a|^2 \leq \|f\|^2 \leq  C_2\rq{} \sum_a |c_a|^2
 \end{align}
 for $ C_1\rq{}  = C_1   \|\hat\phi\|_0^2 $ and $C_2\rq{} = C_2 
 \|\hat\phi\|_\infty^2$. \\

Conversely, suppose  that $\{\phi(\cdot - a)\}_{a\in A}$ is a  Riesz basis  for $PW_\Omega(\Bbb R^d)$ with   constants $m_\phi, M_\phi>0$.  We prove that $\{e_a\}_{a\in A}$ is a Riesz basis for $L^2(\Omega)$. 
Let $u\in L^2(\Omega)$ and  
define $f = \breve{u} \ast \phi$. (Here, $\breve{u}$ is the Fourier inverse function of $u$.) The Fourier transform $\hat f$ has   support in $\Omega$ and   by the assumptions we have    $f = \sum_a c_a \phi(\cdot -x)$ for some $\{c_a\}_{a\in A} \in l^2(A)$, such that    

\begin{align}\label{another-inequality}
m_\phi\sum_a |c_a|^2 \leq \|f\|2 \leq  M_\phi \sum_a |c_a|^2 .  
\end{align}
To show that   $u = \sum_a c_a e_a$ and the Riesz  inequality (\ref{exponential-frame-inequality}) holds for  $u$,  once again 
 we let $B$ be a finite subset of $A$. Then

\begin{align}\notag
\|u-\sum_{a\in B} c_a e_a\| &= \int_\Omega |u(x) - \sum_{a\in B} c_a e_a(x)|^2 dx \\\notag
&= \int |\hat \phi(x)|^{-2} \left|u(x) \hat \phi(x) -  \sum_{a\in B} c_a \hat \phi(x) e_a(x)\right|^2 dx \\\notag
&\leq  \|\hat\phi\|_0^{-2} \|u \hat \phi - \sum_{a\in B} \hat \phi e_a\|^2 \\\notag
&= \|\hat\phi\|_0^{-2} \| \breve{u}\ast \phi -\sum_{a\in B}  \phi(\cdot - a)\|^2\\\notag
& = \|\hat\phi\|_0^{-2} \| f-\sum_{a\in B}  \phi(\cdot - a)\|^2 .
\end{align} 
 
 The above estimation holds for any finite subset of $A$. Therefore, since  $f= \sum_{a\in A}  \phi(\cdot - a)$, we conclude that  $u=\sum_{a\in A} c_a e_a$. \\
 
To prove 
the  inequality (\ref{exponential-frame-inequality}) 
for $u$, first we apply 
the upper  estimation in  (\ref{another-inequality}) as  follows.  
  
  \begin{align}\notag
  \|u\|^2   \leq \|\hat\phi\|_0^{-2} \| u \hat \phi   \|^2= \|\hat\phi\|_0^{-2}  \| \breve{u}\ast \phi\|^2 \leq 
 M_\phi  \|\hat\phi\|_0^{-2}    \sum_a |c_a|^2.
  \end{align}
 (Recall that $\breve{u}\ast \phi\ =f$ and  $*$  denotes the convolution operation.)
With a similar calculation, one can show that 
$$\|u\|^2\geq m_\phi\|\hat\phi\|_\infty^{-2}   \sum_a |c_a|^2. $$

This  completes the proof of  (\ref{exponential-frame-inequality}) for the function $u$ with the Riesz basis constants $  M_\phi\|\hat\phi\|_0^{-2}   $ and $
m_\phi \|\hat\phi\|_\infty^{-2}   $.  Hence, the proof of the theorem is completed. 
\end{proof}

 
 \subsection{Application}
 
In Theorem \ref{equivalence} we observed that for 
  $\hat\psi = \chi_{\Omega}$ with $|\Omega|=1$, the   Paley-Wiener space $PW_{\Omega}(\Bbb R^d)$ is exactly  the reproducing kernel Hilbert space $L^2(\RR)\ast \psi$ and  A-translates of $\psi$ form an ONB (orthonormal basis)  for  this space if and only if $\{e_a\}$ is an ONB for $L^2(\Omega)$.   Motivated by this observation,   here we study  the cases where   for $\psi\in L^2(\Bbb R^d)$, the Hilbert space $L^2(\Bbb R^d)\ast \psi$   has a translate orthonormal basis or a translate Riesz basis.  We will narrow our attention to     the situation   where  $\psi$ is a band-limited function (i.e., it has bounded Fourier transform support),  and 
  $L^2(\Bbb R^d)\ast \psi$ is a Paley-Wiener space.  
  First we need the following lemma. 
\begin{lemma} 
For any $\psi\in PW_\Omega$,  there holds  $PW_\Omega= L^2(\Bbb R^d)\ast \psi$, provided    $0<\|\hat\psi\|_0<\infty$. 
\end{lemma}

\begin{proof} Since  $\psi$ has Fourier support in $\Omega$,  then by the Fourier transform      $L^2(\RR^d)\ast \psi \subseteq PW_\Omega$. 
 For the converse, let 
 $f\in PW_\Omega$. By  our assumption on $\hat\psi$, we have  $\|\hat f/\hat \psi\|\in L^2(\Omega)$ with $\|\hat f/\hat \psi\|\leq \|\hat f\| / \|\hat\psi\|_0$.  
 This implies that for some $g\in L^2(\RR^d)$, 
  $f= g\ast \psi$,   and hence the proof of the lemma is completed. 
\end{proof}


The following result is a direct consequence of the preceding  lemma. 
 
 \begin{corollary}\label{translate basis for rkHs}  
 Let $\phi, \psi\in PW_\Omega$ such   that 
  $0<\|\hat\phi\|_0\leq \|\hat\phi\|_\infty<\infty$  and $0<\|\hat\psi\|_0<\infty$.  Then set
    $A$ is a Riesz basis spectrum for $\Omega$ if and only if 
 $\{L_a\phi\}_{a\in A}$ is a Riesz basis for    $L^2(\Bbb R^d)\ast \psi$.   
 \end{corollary}

 \begin{remark}
  A trivial consequence of the Corollary \ref{translate basis for rkHs} is that   for any  $\phi\in PW_\Omega(\RR^d)$ with $0<\|\hat\phi\|_0\leq \|\hat\phi\|_\infty<\infty$, the system   $\{L_a\phi\}_{a\in A}$ is a Reisz basis  for $L^2(\RR^d)\ast \phi$, if the set $A$ is a Riesz spectrum for $\Omega$. 
  
 \end{remark}

  \section{Exponential frames}\label{second part}

 The study of exponential frames was initiated by Duffin and Schaeffer in their work on non-harmonic Fourier series (\cite{DS}).  
The existence of exponential frames on $L^2(\Omega)$, for  Lebesgue measurable set  $\Omega$  in $\Bbb R$,  is 
 also known to be equivalent to the sampling and interpolation problems on the Paley-Wiener space $PW_\Omega(\Bbb R)$ (see e.g. \cite{Y, Lan, Ja}). For the most recent work for sampling and interpolation of Paley-Wiener (band-limited) functions on  the locally compact abelian groups see e.g.  \cite{Gro, Ga}, and for the results on non-commutative settings see e.g.    \cite{CM}. In this section,  we 
 show that the existence of exponential frames on $L^2(\Omega)$ is equivalent to the existence of  translate basis  for Paley-Wiener spaces. 

 \begin{definition}
Given  $A\subset \RR^d$ and  $\phi \in L^2({\Bbb R}^d)$, we  say that the translation system $\{\phi(\cdot-a)\}_{a \in A}$ is a  {\it frame} for $PW_\Omega({\Bbb R}^d)$  if  there exist  positive  and finite  constants $C_1\leq  C_2$ such that for any   $f \in PW_\Omega({\Bbb R}^d)$   there holds

 $$C_1  \|f\|^2\leq  \sum_{a\in A}|\langle f, \phi(\cdot-a)\rangle|^2 \leq C_2 \|f\|^2.$$

Similarly,   $\{e_a\}_{a\in A}$ is a frame for $L^2(\Omega)$ if  there exist positive  and finite  constants $C_1\leq  C_2$  such that for any  given $f \in L^2(\Omega)$   

$$C_1  \|f\|^2\leq  \sum_{a\in A}|\langle f, e_a\rangle|^2 \leq C_2 \|f\|^2 .$$
A frame is called Parseval if $C_1=C_2$.

\end{definition}

For a necessary and sufficient conditions for  translate frames (resp. Riesz basis) with the spectrum set $A\subseteq \ZZ$ see e.g. \cite{CCK}. \\

Notice that any Riesz basis is a frame, but the converse dose not always hold.   For examples of  frames containing no Riesz bases we  invite the reader  to  see the monograph  \cite{OC} by Christensen. 
To state the first result of this section, we  
  need the following definition. For $\phi\in PW_\Omega(\RR^d)$   let 
    $$ E_\phi= \{ x\in \Omega: \hat\phi(x)\neq 0\}.$$

\begin{theorem}\label{equivalence3} Given $A\subset \Bbb R^d$ and  $\Omega\subset \Bbb R^d$, 
the system    $\{e_a\}_{a\in A}$ is a frame   for $L^2(\Omega)$ if and only if  for any  $\phi\in PW_\Omega(\Bbb R^d)$   the system 
 ${ \{ \phi(x-a) \} }_{a \in A}$ is frame  for $PW_\Omega({\Bbb R}^d)$,  provided that there exist positive  and finite constants $m\leq  M$ such that  $m\leq  |\hat\phi(x)|^2 \leq M$   for a.e. $x\in E_\phi$.  
\end{theorem}  

\begin{proof} 
 For   $f\in PW_\Omega(\Bbb R^d)$, put   $u:=\hat f  \hat\phi \chi_{E_\phi}$. By the assumption on $\hat\phi$, 
  we will then  have 
  \begin{align}\label{frame3}
 m\|f\|^2 \leq \|u\|^2\leq M \|f\|^2,
 \end{align}
 
 and by the Parseval identity, 

\begin{align}\label{frame1}
\sum_a\left|\langle f, L_a\phi\rangle\right|^2 &= \sum_a \left|\langle u, e_a\rangle\right|^2. 
\end{align}
If $\{e_a\}_{a\in A}$ is a frame for $L^2(\Omega)$ with constants $C_1$ and $C_2$, then by (\ref{frame1})  

\begin{align}\label{frame2}
C_1\|u\|^2\leq \sum_a\left|\langle f, L_a\phi\rangle\right|^2 \leq C_2\|u\|^2 .
\end{align}
  
A combination of the inequalities in (\ref{frame2}) and (\ref{frame3}) proves that 
    $\{L_a\phi\}$ is a frame for $PW_\Omega(\Bbb R^d)$.  Conversely, assume that  $\{L_a\phi\}$ is a frame for $PW_\Omega(\Bbb R^d)$ and $u\in L^2(\Omega)$. 
    Put $h:= \chi_{E_\phi}\cfrac{u}{\bar{\hat\phi}}$. By the assumptions, 
     $h$ belongs to $L^2(\Omega)$. If we let  $f$ denote the inverse Fourier transform of $h$, i.e, $\hat f =h$, then 
    \begin{align}\label{frame4}
    \sum_a \left|\langle f, L_a\phi\rangle\right|^2 = \sum_a \left|\langle h,  \hat\phi e_a\rangle\right|^2 = \sum_a \left|\langle u,  e_a\rangle\right|^2.
    \end{align}
 If $C_1$ and $C_2$  are the  frame constants for  $\{L_a\phi\}$, 
then by  (\ref{frame4}) we will have 
    \begin{align}\label{frame5}
   C_1\|f\|^2\leq   \sum_a \left|\langle u,  e_a\rangle\right|^2 \leq C_2 \|f\|^2 .
    \end{align}
    
From the other hand we have 
    $\|f\|= \|h\| = \left\|\chi_{E_\phi}\cfrac{u}{\bar{\hat\phi}}\right\| $ and    $m\leq  |\hat\phi(x)|^2 \leq M$   for a.e. $x\in E_\phi$.  
    Therefore
 
    \begin{align}\label{frame6}
 M^{-1} \|u\|^2\leq  \|f\|^2 \leq m^{-1} \|u\|^2 .
    \end{align}
    By interfering (\ref{frame6}) in  (\ref{frame5}) the frame condition holds for  $\{e_a\}$, and this completes the proof of the theorem.
\end{proof}

\begin{remark}Note that what distinguishes Theorem \ref{equivalence3} from \ref{equivalence2} is that for frames we  only require  $\hat\phi$ to be non-vanishing on  a subset of 
   $\Omega$. 
   \end{remark}



In the following  we  give an example of a frame with a  smooth generator. 
 
 \begin{example} Let $u$ be a bump function on $\Bbb R$ with compact support $[0,1]$. Take any compact set $\Omega\subseteq [0,1]$ such that $u$ is away from zero on $\Omega$ (i.e.,  $\inf_{x\in \Omega} |u(x)|>0$). Let $\breve{u}$ denote the inverse Fourier transform of $u$.    Put   $\phi:= \breve{u}\ast \breve{\chi_\Omega}$. 
 Then $\phi$ 
  is smooth  with $E_\phi= \Omega$, and $\{\phi(\cdot -n)\}_{n\in \Bbb Z}$ forms  a frame for $PW_\Omega(\Bbb R)$ if and only if $A$ is a frame spectrum for $\Omega$.  
 \end{example}

 We conclude this section with the following corollary.

\begin{corollary}
Let $\Omega\subset \Bbb R^d$ be a set with finite positive measure. Then a set $A$ is a Parseval frame spectrum for $L^2(\Omega)$ if and only if for any $\phi\in PW_\Omega(\Bbb R^d)$ the system of  translates $\{\phi(\cdot-a)\}_A$ is a Parseval frame, provided that $\hat \phi$ is nowhere zero on $\Omega$.  
\end{corollary}

 \section{Paley-Wiener valued Gabor systems}\label{Paley-Wiener valued Gabor systems}
 
 In this section, we shall apply the results of  previous sections to characterize vector-valued Gabor systems that are  orthonormal basis for a class of vector-valued signals. Before we state our results, we need to introduce some notations here. 
 \\
 
   For  given measurable space $(X, \mu)$  and Hilbert space $Y$,  the space $\mathcal L: = L^2(X,Y, \mu)$ is defined as class of all equivalent and measurable functions $F: X\to Y$ for which

 $$\|F\|_{\mathcal L}^2= \int_x \|F(x)\|_{Y}^2 d\mu(x) <\infty .$$
 
 $ \mathcal L$ is a Hilbert space with the norm $\|\ \|_{\mathcal L}$  and  the inner product

 $$\langle F, G\rangle_{\mathcal L}  = \int_X \langle F(x), G(x) \rangle_Y  \ d\mu(x) .$$ 
 To avoid any confusion, in the sequel, we shall use subscripts for all inner products for the Hilbert spaces. 
  
 \begin{lemma}\label{mixed orthonormal bases} Let $(X,\mu)$ be a measurable space,  and $\{f_n\}_n$ be an orthonormal basis for $L^2(X):=L^2(X, d\mu)$. Let 
$Y$ be a Hilbert space and $\{g_m\}_m$ be a  family in  $Y$. For any $m, n$ and $x\in M$ define $G_{m,n}(x) := f_n(x) g_m$. Then $\{G_{m,n}\}_{m,n}$ is an orthonormal basis for the Hilbert space $\mathcal L$ if and if $\{g_m\}_m$ is an orthonormal basis for $Y$. 
\end{lemma}
 
     \begin{proof}  For any $m, n$ and $m\rq{}, n\rq{}$ we have the following. 
     
    \begin{align}\label{orthogonality-relation}
    \langle G_{m,n}, G_{m\rq{},n\rq{}} \rangle_{\mathcal L}  &= \int_X \langle f_m(x) g_n, f_{m\rq{}}(x) g_{n\rq{}} \rangle_Y \  d\mu(x)\\\notag
    & = \langle f_m,f_{m\rq{}}\rangle_{L^2(X)} \langle g_n, g_{n\rq{}}\rangle_Y \\\notag
    &=\delta_{m,m\rq{}}  \langle g_n, g_{n\rq{}}\rangle_Y
    \end{align}
 
 This shows  that 
the orthogonality of  $\{G_{m,n}\}_{m,n}$  is equivalent to the orthogonality of $\{g_m\}_m$. 
  And,  $\|G_{m,n}\| =1$ if and only if $\|g_n\|=1$.  \\
  
  Let $\{g_m\}_m$ be an orthonormal basis for $Y$. 
To prove the completeness of  $\{G_{m,n}\}$ in  $\mathcal L$, let  
 $F\in  \mathcal L$ such that $\langle F, G_{m,n}\rangle_{\mathcal L}  =0$, $\forall \ m, n$. We show that $F=0$.     By the definition of inner product we have 

\begin{align}\label{inner-product}
0=\langle F, G_{m,n}\rangle_{\mathcal L} &=\int_X \langle F(x), G_{m,n}(x)\rangle_Y d\mu(x)\\\notag 
&= \int_X \langle F(x), f_m(x) g_n\rangle_Y d\mu(x) \\\notag
&= \int_X \langle F(x), g_n\rangle_Y \overline{f_m(x)} d\mu(x) \\\notag
&= \langle A_n, f_m\rangle
\end{align}
where  
$$A_n: X\to \CC; \ \  x \mapsto \langle F(x), g_n\rangle_Y. $$
$A_n$ is a measurable function and lies in $L^2(X)$ with $\|A_n\|\leq  \|F\|$. Since $\langle A_n, f_m\rangle_{L^2(X)}=0$ for all  
 $m$ (see above), then  $A_n=0$ by the completeness of $\{f_m\}$.  From the other hand, by  the definition of $A_n$  we have    $\langle F(x), g_n\rangle_{Y}=0$ for a.e. $x\in X$.   Since   $\{g_n\}$   is complete in  $Y$,  then $F(x)= 0 $ for a.e. $x\in X$, as we desired.    \\
   
   Conversely, assume that $\{G_{m,n}\}_{m,n}$ is an orthonormal basis for the Hilbert space $\mathcal L$. Therefore by (\ref{orthogonality-relation}), $\{g_m\}$ is an  orthonormal system. We prove that if for $g\in Y$ and $\langle g, g_m\rangle=0$ for all $m$, then $g$ must be zero. For this,  for any $n$ let us define 
   the map 
   
   $$B_n: X\to Y; \ \ x\mapsto f_n(x)g.$$
   Then $B_n$ is measurable and it belongs to $\mathcal L$ with $\|B_n\|_{\mathcal L} = \|g\|_Y$. By expanding $B_n$ in terms of $\{G_{m,n}\}$, 
  \begin{align}
  B_n&=\sum_{n\rq{},m} \langle B_n, G_{m,n\rq{}}\rangle_{\mathcal L} G_{m,n\rq{}}\\\notag
  &= \sum_{n\rq{},m} \langle f_n, f_{n\rq{}}\rangle_{L^2(X)} \langle g, g_m\rangle_Y G_{m,n\rq{}}\\\notag
    &= \sum_{m}   \langle g, g_m\rangle_Y G_{n,m}
 \end{align}
   
   By the assumption that 
   $ \langle g, g_m\rangle_Y=0$, we will have $B_n=0$. This implies that $B_n(x)= f_n(x) g=0$ for a.e.  $x$. Since, $f_n\neq 0$, then  $g$ must be zero, and hence we are done.

\end{proof}

  We conclude this paper with  another description of translate orthonormal bases in terms of {\it vector-valued Gabor systems}  introduced in  Theorem \ref{vector-valued bases}. The result is   a consequence of 
   Lemma \ref{mixed orthonormal bases}  along  Theorem \ref{equivalence}.    First we need a definition here.
   
 \begin{definition} Given $A$, $\Omega$, and $u\in PW_\Omega(\RR^d)$, for any $a, b\in A$,  the vector-valued time-frequency translate of $u$ along $a, b\in A$ is given by $u_{a,b}: \Omega\to PW_\Omega(\RR^d)$, \ $ \alpha\mapsto e_b(\alpha) T_a u$.  The system
 \begin{align}\label{product-base}
 \mathcal G(A, \Omega, u):= \{G_{a,b}: \   G_{a,b}(\alpha):= e_b(\alpha) T_a u, \ a,b\in A\}
 \end{align}
    is called the  corresponding vector-valued Gabor  or {\rm Weyl-Heisenberg} system  for the \lq\lq{}window\rq\rq{} signal $u$. For more on traditional Weyl-Heisenberg frames we refer the reader to \cite{Cas-Chris-Jans01, Casazza-Weyl-Heis}. 
\end{definition}

  \begin{theorem}\label{vector-valued bases} Let $|\Omega|=1$. 
  Define $\mathcal L= L^2(\Omega, PW_\Omega)$ to  be the Hilbert space of all measurable functions $F: \Omega \to PW_\Omega$ with finite    norm  $\|F\|_{\mathcal L}:=\int_\Omega \|F(\alpha)\|_{L^2(\RR^d)}^2  d\alpha$.  
 Then for any $u\in PW_\Omega$,  $\{T_au\}_{a\in A}$ is  an orthonormal basis for $PW_\Omega$ if and only if 
the vector-valued Gabor  system   $\mathcal G(A, \Omega, u)$ is an orthonormal basis for $\mathcal L$, provided that 
 $\hat u$ is  nowhere zero on $\Omega$. 
  \end{theorem}

  \begin{proof}  Let $u\in PW_\Omega$ with  $\hat u$   nowhere zero on $\Omega$. Let $\{T_{a}u\}_{a\in A}$ be an orthonormal basis for $PW_\Omega$.  Then by Theorem \ref{equivalence2}, the exponentials $\{e_a\}$ form an orthonormal basis for $L^2(\Omega)$, and hence 
  $\mathcal G(A,\Omega, u)$  given in  (\ref{product-base}) is an orthonormal basis for $\mathcal L$
  by  Lemma \ref{mixed orthonormal bases}. 
To prove the converse, let  $\mathcal G(A, \Omega, u)$ be an orthonormal basis for $\mathcal L$.   The orthogonality of the family $\{ T_a u\}$ follows from the following simple relation: For  $a, a' \in A$   and any $b\in A$
    $$ \delta_{a,a'}= \langle G_{a,b} , G_{a',b}\rangle_{\mathcal L}  = \langle T_a u, T_{a'} u\rangle_{L^2(\RR^d)}.$$ 
    
  To show     $\{ T_a u\}$  spans  $PW_\Omega$, let $g\in PW_\Omega$ such that $\langle g, T_a u\rangle_{L^2(\RR^d)}= 0$ for all $a\in A$.  We  show that $g=0$. For this,  fix $b$ and define $$W_b: \Omega\to PW_\Omega; \ \alpha\mapsto   e_b(\alpha)  g. $$
  It is clear that $W_b\in \mathcal L$ with $\|W_b\|_{\mathcal L}= \|g\|_{L^2(\RR^d)}$, and  for any $a, b'$ we have  
  
 \begin{align}\label{ortho}
  \langle W_b, G_{a,b\rq{}}\rangle_{\mathcal L} = \int_\Omega \langle g  e_b(\alpha) ,  e_{b\rq{}}(\alpha) T_a u \rangle_{L^2(\RR^d)}  \ d\alpha= \langle g, T_au\rangle_{L^2(\RR^d)} \int_\Omega e_b(\alpha)\overline{e_{b\rq{}}(\alpha)}d\alpha .
 \end{align} 
 
 Since 
 $\langle g, T_au\rangle_{L^2(\RR^d)}=0$ and
  the system  $\mathcal G(A,\Omega, u)$ is complete in $\mathcal L$,  then   (\ref{ortho}) implies that   $W_b (\alpha) = e_b(\alpha) g=0$ for almost every $\alpha$. But  $e_b\neq 0$, therefore $g$ must be zero, and hence we finish the proof. 
  \end{proof}
   
We conclude this paper with the following remark. 

\begin{remark} A special case of 
 Theorem \ref{vector-valued bases} is when $A=\ZZ^d$. This case has already been studied, e.g., in 
  \cite{B}. The author  showed that 
a necessary and sufficient condition for $\{T_ju\}_{j\in \ZZ^d}$ to be an orthonormal basis for its span is that  for almost every $x\in [0,1]^d$ 
 
 \begin{align}\label{periodic}
 \sum_{j\in \ZZ^d} |\hat u(x+j)|^2 = 1
 \end{align} 
 
 As we mentioned earlier, the similar techniques used in  \cite{B} do not apply to our situation where  $A$ is any random set. Therefore in our paper  we used different   techniques to overcome the problem and could prove a characterization of translate bases (ONB, Riesz, frame) in terms of exponential bases on a domain.  
     For the equivalent conditions  for orthonormal bases,  frame, and Riesz basses of discrete  translations   in terms of the periodic function $x\mapsto \sum_{j\in \ZZ^d}  |\hat u(x+j)|^2$     on $\Bbb R^n$ and locally compact abelian  groups we refer to \cite{B} respectfully  \cite{BHM}.  Relevant  results for  these bases in non-commutative settings, in particular, the Heisenberg group,   have   recently   been  shown  by the second author et  al.  in   \cite{BHM, CMO}. 
  
 \end{remark}

\end{document}